\documentclass[reqno]{amsart}

\usepackage{amsfonts}
\usepackage{amsmath,amssymb,amscd,graphicx,epsfig,psfrag,textcomp,empheq,yfonts,mathrsfs,bbm,bm}
\everymath{\displaystyle}

\newtheorem{theorem}{Theorem}[section]
\newtheorem{lemma}[theorem]{Lemma}
\newtheorem{corollary}[theorem]{Corollary}

\newtheorem{proposition}[theorem]{Proposition}
\newtheorem{prop-def}[theorem]{Proposition-Definition}
\newtheorem{lem-defi}[theorem]{Lemma-Definition}
\newtheorem{definition}[theorem]{Definition}

\newtheorem{prop-defi}[theorem]{Proposition-Definition}

\newcommand{\Cc}{{\mathcal C}}
\newcommand{\Dd}{{\mathcal D}}

\newcommand{\Ff}{{\mathcal F}}

\newcommand{\Ll}{{\mathcal L}}
\newcommand{\Mm}{{\mathcal M}}

\newcommand{\Pp}{{\mathcal P}}
\newcommand{\Qq}{{\mathcal Q}}

\newcommand{\Ww}{{\mathcal W}}

\newcommand{\CC}{{\bf C}}

\newcommand{\SN}{{\bf S}}
\newcommand{\QQ}{{\bf Q}}
\newcommand{\DD}{{\bf D}}

\newcommand{\Hhh}{{\bf h}}

\newcommand{\NN}{{\bf N}}
\newcommand{\PP}{{\mathbf P}}
\newcommand{\RR}{{\bf R}}

\newcommand{\Sb}{{\bf S}}

\newcommand{\MB}{{\mathbf M}}

\newcommand{\rr}{{\mathbf{r_1}}}
\newcommand{\rl}{{\mathbf{r}}}
\newcommand{\Tr}{{\mathbf{\tilde r_1}}}
\newcommand{\rrr}{{\mathbf{r_2}}}
\newcommand{\uu}{{\mathbf{u_1}}}
\newcommand{\uuu}{{\mathbf{u_2}}}

\newcommand{\vv}{{\mathbf{v}}}
\newcommand{\x}{{\mathbf{x}}}
\newcommand{\y}{{\mathbf{y}}}
\newcommand{\X}{{\mathbf{X}}}
\newcommand{\Y}{{\mathbf{Y}}}

\DeclareMathOperator{\ep}{\varepsilon}

\newlength{\upshgt}
\DeclareMathOperator{\Dosilon}{\raisebox{\upshgt}{\scalebox{1}[-1]{$\Upsilon$}}}

\newcommand{\aG}{{\mathfrak a}}

\newcommand{\BB}{{\mathbf{B}}}

\numberwithin{equation}{section}

\graphicspath{{Figures/}}  

\begin{document}

\title{Characterization of The Generic Unfolding of a Weak Focus}

\pagestyle{headings}
\noindent 
\author{  by\\ W. Arriagada-Silva \ \,\\ \,\\ \,\\ \, \\ \,\\ \, \\ \,\\ \,  }

\begin{abstract}
In this paper we give a geometric description of the foliation of a generic real analytic family unfolding a real analytic vector field with a weak focus at the origin, and show that two such families are orbitally analytically equivalent if and only if the families of diffeomorphisms unfolding the complexified Poincar\'e map of the singularities are conjugate.    Moreover, by shifting the leaves of the formal normal form in the blow-up (quasiconformal surgery) by means of a fibered transformation along a convenient complex cross-section, one constructs an abstract manifold of complex dimension 2 equipped with an elliptic holomorphic foliation whose monodromy map coincides with a given family of admissible diffeomorphisms.
\end{abstract}

\noindent 
\bibliographystyle{plain}

\address{{\it W. Arriagada:}  D\'epartement de math\'ematiques et de statistique,
Facult\'e des arts et des sciences -Secteur des sciences, Universit\'e de Montr\'eal\\
succ. Centre-ville\\
Montr\'eal, Qc\\
H3C 3J7.}
\email{arriaga@DMS.Umontreal.CA}


\date{\today}

\maketitle
\markboth{ W. Arriagada Silva}{Characterization of the generic unfolding of a weak focus}

\section{Introduction.}\label{intro}
A one-parameter family of real analytic planar systems unfolding a weak focus is an elliptic real analytic one-parameter $\ep\in\RR$ dependent family linearly equivalent to a family of planar differential equations
\begin{equation}
\left.
\begin{array}{lll}
\dot x &=& \alpha(\ep)x -\beta(\ep)y + \sum_{j+k\geq 2} b_{jk}(\ep)x^jy^k\\
\dot y &=& \beta(\ep)x + \alpha(\ep)y + \sum_{j+k\geq 2} c_{jk}(\ep)x^jy^k,
\end{array}\label{weakfocus00}
\right.
\end{equation}
for real time, and with $\alpha(0)=0$ and $\beta(0)\neq 0.$   After rescaling the time $t\mapsto\beta(\ep)t$ we can suppose $\beta(\ep)\equiv 1.$    The family \eqref{weakfocus00} is called ``generic'' if $\alpha'(0)\neq 0.$    The genericity allows to take $\alpha$ as the new parameter, so that the eigenvalues become $\ep+i$ and $\ep-i,$ respectively.\\

When the order is one, a weak focus of a real analytic vector field corresponds to the coalescence of a focus with a limit cycle, and the generic family \eqref{weakfocus00} is then a family with a generic Hopf bifurcation, whose foliation is described by the unfolding of the Poincar\'e map or monodromy $\Pp_{\ep}:(\RR^+,0)\to (\RR^+,0)$ of the system.    It is well known that the germ of the Poincar\'e return map or monodromy  is well defined and analytic, and can be extended to an analytic diffeomorphism 
\begin{equation}
\begin{array}{lll}
\Pp_{\ep}:(\RR,0)\to (\RR,0).
\end{array}\label{usupm}
\end{equation}
A question that arises naturally is whether the germ of the monodromy map defines the analytic equivalence class of the real foliation.    The natural way to answer this question is via complexification (cf. \cite{bcl}).    The right hand side of the complexified system is now defined by an analytic family of vector fields
\begin{equation}
\begin{array}{lll}
v_{\ep}(x,y)= P(x,y)\frac{\partial}{\partial x} + Q_{\ep}(x,y)\frac{\partial}{\partial y}
\end{array}\label{weakfocus0}
\end{equation}
that satisfies
\begin{equation}
\begin{array}{lll}
P_{\ep}(x,y) = \overline{Q_{\overline{\ep}}(\overline y,\overline x)},
\end{array}\label{realchar}
\end{equation}
where $x\mapsto \overline x$ is the complex conjugation.   The time is complexified as well and the domain of the parameter is now a standard open complex disk noted $V\in\CC.$    After complexification, the real plane can be written in a rather simple way: it corresponds to the surface $\{x=\overline y\}.$    The Poincar\'e map of the complexified system (parametrized with $x$-coordinate) is defined as the second iterate of the holonomy $\Qq_{\ep}$ along the loop $\RR P^1$ (the equator of the exceptional divisor) of the foliation after standard blow-up (cf. \cite{YY}), where the standard affine coordinates on the projective line $\CC P^1$ are given by formulas with real coefficients, hence defining correctly the real projective equator $\RR P^1\subset\CC P^1,$ see Figure \ref{monoPoincare}.   Blowing down the foliation, the Poincar\'e map is defined on the $1$-dimensional complex cross-section $\{x=y\},$ and the usual real germ \eqref{usupm} of the planar system is defined on $\{x=y\}\cap\{x=\overline y\}.$\\

\noindent{\emph{Notation.}}    The cross section $\{x=y\}$ is noted $\Sigma$ and is parametrized with the complex coordinate $x.$\\

The complex description of the monodromy immediately allows to prove its analyticity, even at the origin.    The monodromy is then a real holomorphic germ of resonant diffeomorphism with a fixed point of multiplicity 3 at the origin, which corresponds in the limit $\ep=0$ to the coalescence of a fixed point with a 2-periodic orbit: the fixed point and periodic orbit bifurcate in a generic unfolding.
\begin{figure}[!h]
\begin{center} 
{\psfrag{=}{\huge{$\Sigma$}}
\psfrag{a}{\huge{$w^*$}}
\psfrag{q}{\huge{$\Qq(w^*)$}}
\psfrag{e}{\huge{$\RR P^1$}}
\psfrag{c}{\huge{$\CC P^1$}}
\psfrag{z}{\huge{$ $}}
\psfrag{w}{\huge{$ $}}
\psfrag{x}{\huge{$x^*$}}
\psfrag{p}{\huge{$\Pp(x^*)$}}
\psfrag{qq}{\huge{$\Qq(x^*)$}}
\psfrag{o}{\huge{$0$}}
\psfrag{R}{\huge{$\RR$}}
\scalebox{0.40}{\includegraphics{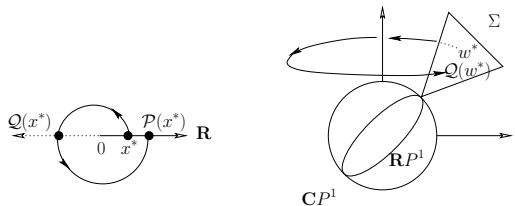}}
}
\end{center}
\caption{\label{monoPoincare} \small{The complexification of the real line and its blow-up.}}
\end{figure}

\noindent{\emph{The equivalence problem.}}    It is known that the problem of orbital equivalence for germs of analytic vector fields with a resonant saddle point is reduced to the conjugacy problem for germs of diffeomorphisms (the holonomy map) with a fixed point at the origin and multiplier on the unit circle (cf. \cite{MM}).    In the non-resonant case, the statement holds as well, as was shown by R. P\'erez-Marco and J.-C. Yoccoz (cf. \cite{peyo}).    Furthermore, this result has been extended to generic analytic families unfolding a resonant saddle point (cf. \cite{cri-rou}).
\begin{definition}
An analytic orbital equivalence (resp. conjugacy) between two analytic germs of families unfolding germs of analytic vector fields (resp. diffeomorphisms) is said to be ``real'', when it leaves invariant the real plane (resp. the real line) for real values of the parameter.
\end{definition}
\noindent In this paper, we show that the equivalence problem for \eqref{weakfocus0} can be reduced to the conjugacy problem for the associated family unfolding the complexified Poincar\'e map, respecting the underlying real foliation.   More precisely,
\begin{theorem}\label{mmexten-unfolding}
Two germs of generic families of real analytic vector fields \eqref{weakfocus0} are analytically orbitally equivalent by a real change of coordinates, if and only if the families unfolding their Poincar\'e maps are analytically conjugate by a real conjugacy.
\end{theorem}

\noindent{\emph{The realization problem.}} A second related problem consists in recovering the germ of the analytic foliation when the Poincar\'e map has been prescribed.    This is the problem of ``realization''.    
We give an answer to this problem by means of the desingularization technique and quasiconformal surgery, as suggested by Y. Ilyashenko (cf. \cite{Yl}): for every $\ep\in V,$ one constructs, with the help of an adequate partition of the unity depending only on the argument of the coordinate induced in the separatrix (the exceptional divisor) by the desingularization process, a fibered $C^{\infty}$ transformation or ``sealing map'' defined on a semi-disk.    By shifting the leaves of the normal form with the help of the sealing map, one obtains a $C^{\infty}$ foliation over the product $\CC^*\times\DD_{r},$ and an integrable almost complex structure, making the foliation actually holomorphic.   The almost complex structure extends smoothly along the vertical axis, because the sealing is, by definition, infinitely tangent to the identity.    It remains integrable after the extension.   The Newlander-Nirenberg Theorem yields a $C^{\infty}$ real system of coordinates (depending analytically on $\ep)$ that straightens the almost complex structure, and therefore, the $C^{\infty}$ foliation in a holomorphic foliation that extends by Riemann along the vertical axis.    The blow down of such a foliation is the required generic elliptic family.\\

In this second part we deal with formal normal forms.    Normal form theory provides an algorithmic way to decide
whether two germs of planar vector fields are equivalent under a $C^{N}$-change of coordinates (cf. \cite{crnf}), in which case, the normal forms are polynomial.   However, in the analytic case, the
formal change of coordinates to normal form generically diverges (cf. \cite{ynf}). An
explanation of this is found by considering unfoldings of the vector fields
and explaining the divergence in the limit process.    This is a particular manifestation of the so-called Stokes Phenomenon (cf. \cite{Yl}).    The spirit of the general answer is the following (cf. \cite{crnf}).   The dynamics of the
original system is extraordinarily rich to be encoded in the simple dynamics
of the normal form which depends of at most one parameter.   Hence the divergence
of the normalizing series.

\section{Proof of Theorem \ref{mmexten-unfolding}.}
The proof uses basically the classical fact that the holonomy characterizes the differential equation (cf. \cite{MM} and \cite{MR}), plus an additional ingredient: both the equivalence between vector fields and the conjugacy between Poincar\'e maps, must respect the real foliation.\\

By definition, if two families \eqref{weakfocus0} are orbitally equivalent by an analytic change of coordinates $\Psi_{\ep}$ (depending analytically on the parameter), it is always possible to reparametrize the families and suppose that they have the same parameter.   Thus, one direction is obvious: if two families of vector fields are equivalent by real change of coordinates, then the equivalence induces a real analytic return map on the image $\Psi_{\ep}(\Sigma),$ for each value of $\ep$ over a small neighborhood of the origin.    Because the equivalence is real, the image of the real line under the equivalence is a real analytic curve $\Cc\subset\Psi_{\ep}(\Sigma)$ different, in general, to $\RR,$ see Figure \ref{reconpoin}.    Standard transversality arguments and the Implicit Function Theorem show that there exists an analytic local transition map $\pi$ between $\Sigma$ and $\Psi_{\ep}(\Sigma)$ (cf. \cite{wacr}).     By unicity, any real local trajectory passing through a real point in $\Sigma$ intersects the image $\Psi_{\ep}(\Sigma)$ in a real point.   Thus, the transition is real and it sends the curve $\RR$ into $\Cc,$ and the composition $\pi^{-1}\circ\Psi_{\ep}$ provides a real conjugacy between Poincar\'e maps $\Sigma\to\Sigma.$\\
\begin{figure}[!h]
\begin{center} 
{\psfrag{a}{\Huge{$\Sigma$}}
\psfrag{b}{\Huge{$\RR^2$}}
\psfrag{psi}{\Huge{$\Psi_{\ep}$}}
\psfrag{pi}{\Huge{$\pi$}}
\psfrag{ee}{\Huge{$\RR$}}
\psfrag{gg}{\Huge{$\Cc$}}
\psfrag{ff}{\Huge{$\RR$}}
\psfrag{c}{\Huge{$\Sigma$}}
\psfrag{d}{\Huge{$\Psi_{\ep}(\Sigma)$}}
\scalebox{0.32}{\includegraphics{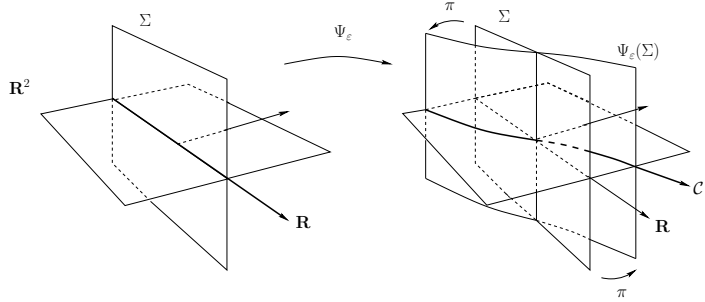}}
}
\end{center}
\caption{\label{reconpoin} \small{The real line and its image by the equivalence $\Psi_{\ep}.$}}
\end{figure}

Let us show the converse.    The conjugacy between the Poincar\'e maps provides a reparametrization, so we can suppose that the parameter is the same for the two families of diffeomorphisms and is henceforth noted $\ep.$    We will suppose that the real conjugacy $\Hhh_{\ep}(x)=\Hhh(\ep,x)$ depends on the $x$-variable and is defined on $\DD_{\rho}\subset\Sigma,$ for every $\ep\in V,$ where $\DD_{\rho}\subset\Sigma$ is the standard open disk of the complex plane, of small radius ${\rho}>0.$     

A theorem on the existence of invariant analytic manifolds (cf. \cite{YY},\cite{MM}) ensures that it suffices to show the theorem for Pfaffian $1$-forms $$\omega_{\ep}, \widehat{\omega}_{\ep} = (\ep+i)xdw - (\ep-i) y(1+xy(...))dx$$ before desingularization.   So if the blow-up space is equipped with coordinates $(X,y)$ and $(x,Y),$ where the standard monoidal map blows down as
\begin{equation}
\begin{array}{lll}
c_1:(X,y)\mapsto(Xy,y),\\
c_2:(x,Y)\mapsto(x,xY)
\end{array}\label{bdown}
\end{equation}
respectively in each direction, the pullback of $\omega_{\ep}$ is defined by $$\omega_1=Xdy-\lambda(\ep) y(1+A_{\ep}(X,y))dX$$ in $(X,y)$ variables, and by $$\omega_2=Ydx-\lambda'(\ep) x(1+A'_{\ep}(x,Y))dY$$ in $(x,Y)$ coordinates, where $A_{\ep}(X,y)=O(Xy)$ and $A'_{\ep}(x,Y)=O(xY)$ depend analytically on the parameter and are holomorphic on a neighborhood $\CC^*\times\DD_{s}$ of the exceptional divisor, for each fixed value of $\ep.$  The numbers $\lambda(\ep)=(\ep-i) / 2i$ and $\lambda'(\ep)=-(\ep+i) / 2i$ are the ratios of eigenvalues of the singular points $(X,y)=(0,0)$ and $(x,Y)=(0,0),$ respectively.    In addition, the coordinates can always be scaled before blow-up, to ensure:
\begin{equation}
\begin{array}{lll}
|A_{\ep}(X,y)|,|A'_{\ep}(x,Y)|< 1/2
\end{array}\label{Func-A}
\end{equation}
in $\CC^*\times\DD_s.$   Notice that in complex coordinates, the section $\Sigma$ is parametrized as $\{X=1\}$ in the $(X,y)$ chart, and as $\{Y=1\}$ in the $(x,Y)$ chart.   Bounded equivalences $\widehat{\Psi}_{\ep}^{c_1},\widehat{\Psi}_{\ep}^{c_2}$ are constructed in $(X,y)$ and $(x,Y)$ variables, in such a way that they are analytic continuations of each other over a neighborhood of the exceptional divisor.

\subsection{The equivalence in the $(X,y)$ chart.}
Take a point $y^*\in\DD_{\rho}.$   A former equivalence $\widehat{\Psi}_{\ep}^{c_1}$ is defined on $\Sigma\times\DD_{\rho}$ by $$\widehat{\Psi}_{\ep}^{c_1}:(1,y^*)\mapsto (1,\Hhh_{\ep}(y^*)).$$   This change of coordinates is extended along a subset of $\Sb^1\times\CC$ in the following way.    Notice that the restriction of the form $\omega_1$ to the cylinder $\RR P^1\times\RR^2$ (noted $\{|X|=1\})$ is non-singular and holomorphic, thus it defines a local holomorphic foliation $\Ff_{\omega_1}$ there.    Consider (cylindrical) solutions to $\omega_1=0$ (the first coordinate is to be parametrized by $X=e^{i\theta},$ $\theta\in[0,2\pi]).$
\begin{lemma}\label{scylsol}
Any (cylindrical) solution $\uu$ to
\begin{equation}
\begin{array}{lll}
\uu'=\lambda(\ep)\uu(1+A(e^{i\theta},\uu)),\quad\theta\in[0,2\pi]
\end{array}\label{diffeq-u_1(theta)}
\end{equation}
satisfies $|\uu(0)|e^{-\theta \left\{|\ep|+\frac{1}{4}\right\}}<|\uu(\theta)|<|\uu(0)|e^{\theta \left\{|\ep|+\frac{1}{4}\right\}},$ for any $\theta\in(0,2\pi].$
\end{lemma}
\begin{proof}
The parameter is written as $\ep=\ep_1+i\ep_2,$ with $\ep_1,\ep_2\in\RR.$    As we consider solutions in $|X|=1,$ the time is parametrized by $t=i\theta,$ and then \eqref{diffeq-u_1(theta)} implies $$d\ln\uu=\frac{1}{2}(\ep-i)(1+A_{\ep}(e^{i\theta},\uu))d\theta.$$    Thus, after taking real parts and using the hypothesis \eqref{Func-A} we get, for $\theta\neq 0:$
\begin{equation*}
\begin{array}{lll}
\left|\ln\left|\frac{\uu}{\uu(0)}\right|\right| &\leq& \frac{1}{2}\int_{0}^{\theta} \{|\ep_1|(1+|Re(A_{\ep})|)+|Im(A_{\ep})|(1+|\ep_2|)\} d\theta\\
&<& \frac{1}{2}\int_{0}^{\theta} \{2|\ep| + 1 /2\} d\theta = \theta \left\{|\ep|+ 1/4\right\},\\
\end{array}
\end{equation*}
and the conclusion follows.
\end{proof}
Put $r=\rho e^{-\pi}.$    We denote by $\SN_{r}$ the set of (cylindrical) solutions $\uu$ to \eqref{diffeq-u_1(theta)} for which there exists $\theta_0\in[0,2\pi)$ such that $\uu(\theta_0)\in\DD_{r}.$ 
\begin{corollary}\label{solutcyl}
If $\uu\in\SN_r,$ then $\uu(0)\in\DD_{\rho},$ provided $|\ep|< 1 /4.$
\end{corollary}
This is how the equivalence is extended.   Choose a point $(e^{i\theta_0},y_0)\in\Sb^1\times\DD_r.$    By definition, the path $\gamma:(e^{i\theta},0)$ is lifted in the leaf of $\Ff_{\omega_1}$ containing $y_0\in\DD_{r}$ as $(e^{i\theta},\uu(\theta)),$ for a certain $\uu\in\SN_{r}$ and $\uu(\theta_0)=y_0.$     By Corollary \ref{solutcyl}, the point $\widetilde y:=\uu(0)$ belongs to $\DD_{\rho}.$    If $\gamma$ is lifted in the leaf of $\Ff_{\widehat{\omega}_1}$ passing through $\Hhh_{\ep}(\widetilde y)$ as $(e^{i\theta},\uuu(e^{i\theta},\widetilde y)),$ with $\uuu(1,\widetilde y)=\Hhh_{\ep}(\widetilde y),$ then we define the analytic change of variables by:
\begin{equation}
\begin{array}{lll}
\widehat{\Psi}_{\ep}^{c_1}:\Sb^1\times\DD_r\to\Sb^1\times\CC,\\
\widehat{\Psi}_{\ep}^{c_1}:(e^{i\theta_0},\uu(\theta_0))\mapsto(e^{i\theta_0},\uuu(e^{i\theta_0},\widetilde y)).
\end{array}\label{PsiZ}
\end{equation}
The change \eqref{PsiZ} respects the transversal fibration given by $X=const.$ and is clearly the restriction of a (unique) holomorphic diffeomorphism conjugating $\Ff_{\omega_1}$ and $\Ff_{\widehat{\omega}_1}$ in a neighborhood of $\Sb^1\times\DD_r.$    Moreover, it extends analytically to $\DD_1\times\DD_r$ (where $\DD_1$ is the standard unit (closed) disk of the $X$-separatrix) by means of the lifting of radial paths
\begin{equation*}
\begin{array}{lll}
\gamma_{X_1}  :  [0,-\log |X_1|]  \to \CC,\quad s \mapsto \gamma_{X_1}(s)=(X_1 e^{s},0)
\end{array}\label{radpathZ}
\end{equation*}
for $0<|X_1|<1.$    In fact, suppose that this curve lifts in the leaves of $\Ff_{\omega_1}$ as $$\gamma_{X_1,y_1}:s\mapsto (X_1e^{s},\rr(s,y_1)),\quad \rr(0,y_1)=y_1,$$ for a given $y_1$ small.    Then the solution $\rr(\cdot,y_1)$ of $\omega_1=0,$ with parameter $0<|X_1|<1,$ and initial condition $\rr(0,y_1)=y_1$ is defined on $[0,-\log|X_1|].$     Actually, the hypothesis \eqref{Func-A} shows that 
\begin{equation}
\begin{array}{lll}
|\rr|\leq |y_1|e^{s\left\{|\ep|-\frac{1}{4}\right\}}< |y_1|,
\end{array}\label{voodoopre}
\end{equation}
whenever $|\ep|<1/4.$   We will suppose that the inverse path of $\gamma_{X_1}$ lifts in the leaf of $\Ff_{\widehat{\omega}_1}$ through the point $(\frac{X_1}{|X_1|},y^0),$ where $y^0$ is small, as 
\begin{equation*}
\begin{array}{lll} 
\gamma_{X_1,y^0}^{-1}: s \mapsto (X_1 e^{-(s+\log|X_1|)},\Tr(s,y^0)),\quad s\in[0,-\log|X_1|].
\end{array}\label{radpathZ-1}
\end{equation*}
Consider the only cylindrical solution $\mathbf{u}_{\mathbf 1,X_1,y_1}$ to \eqref{diffeq-u_1(theta)} satisfying $\mathbf{u}_{\mathbf 1,X_1,y_1}(\arg X_1)=\rr(-\log|X_1|,y_1)$ and define the coordinate $$\widetilde y(X_1,y_1):=\mathbf{u}_{\mathbf 1,X_1,y_1}(0)\in\Sigma.$$    Then, \eqref{voodoopre} proves that $\mathbf{u}_{\mathbf 1,X_1,y_1}\in\SN_r$ if $y_1$ is taken in $\DD_r.$    In this case, Corollary \ref{solutcyl} ensures that $\widetilde y(X_1,y_1)$ belongs to $\DD_{\rho}.$    The equivalence is then defined by
\begin{equation}
\begin{array}{lll}
\widehat{\Psi}_{\ep}^{c_1} : (X_1,y_1)\mapsto(X_1,\rrr(X_1,y_1)),
\end{array}\label{Equivalence-Zdisk}
\end{equation}
with $\rrr(X_1,y_1)=\Tr(-\log|X_1|,\uuu(e^{i\arg(X_1)},\widetilde y(X_1,y_1)))$ $(\uuu$ given in \eqref{PsiZ}).   As the change of coordinates is bounded, the Riemann's removable singularity Theorem implies the existence of a unique holomorphic extension $\widehat{\Psi}_{\ep}^{c_1}$ to $\DD_1\times\DD_r.$\\    

Finally, the change of coordinates \eqref{Equivalence-Zdisk} extends to a subset $$\Dd_1(r) = \{(X,y)\in\CC \times\CC : |X|\geq 1, |Xy|\leq r \}$$ as follows.    Similar arguments as those used above show that the only tangent curve $\rl(\cdot,y_1)$ to $\omega_1$ verifying $\rl(\log|X_1|,y_1)=y_1,$ for a given $(X_1,y_1)\in\Dd_1(r),$ satisfies $$|\rl(0,y_1)|e^{-s\{|\ep|+1/4\}} < |\rl(s,y_1)|,\quad s\in[0,\log|X_1|],$$ so that the initial condition $\rl(0,y_1)$ of the lifting starting at $(\frac{X_1}{|X_1|},\rl(0,y_1))$ belongs to $\DD_r$ provided $|\ep|\leq 3/4.$    Thus, the leaf containing the point $(X_1,y_1)$ intersects the cylinder $\{|X|=1\}$ in a curve $\uu=\uu(\theta)\in\SN_r,$ with $\uu(\arg X_1)=\rl(0,y_1)\in\DD_r.$    By Corollary \ref{solutcyl}, $\uu(0)\in\DD_{\rho}$ and then 
$\widehat{\Psi}_{\ep}^{c_1}(\frac{X_1}{|X_1|},\rl(0,y_1))$ is well defined, where $\widehat{\Psi}_{\ep}^{c_1}$ is the equivalqnce \eqref{Equivalence-Zdisk}.    In $\Ff_{\widehat{\omega}_1},$ the inverse of $\gamma_{X_1}$ is lifted on the leaf passing through the point $\widehat{\Psi}_{\ep}^{c_1}(\frac{X_1}{|X_1|},\rl(0,y_1)).$    The endpoint of this radial lifting defines $\widehat{\Psi}_{\ep}^{c_1}$ on $\Dd_1(r).$

\subsection{The equivalence in the $(x,Y)$ chart.}
If $\DD_2$ is the standard unit (closed) disk of the $Y$-separatrix and $\Dd_2(r) = \{(x,Y)\in\CC\times\CC : |Y|\geq 1, |xY|\leq r \},$ then, in $(x,Y)$ coordinates the equivalence is defined plainly on $(\DD_2^*\times\DD_r)\cup\Dd_2(r),$ by the formula $$\widehat{\Psi}_{\ep}^{c_2}:=\varphi\circ\widehat{\Psi}_{\ep}^{c_1}\circ\varphi^{\circ -1},$$ where $\varphi: (X,y)\mapsto (x,Y)$ is the transition between complex charts.    
Such equivalence is clearly bounded and the Riemann's Theorem yields a unique holomorphic extension $\widehat{\Psi}_{\ep}^{c_2}:(\DD_2\times\DD_r)\cup\Dd_2(r)\mapsto\CC^2.$

It turns out that the two changes of coordinates thus obtained $\widehat{\Psi}_{\ep}^{c_1},\widehat{\Psi}_{\ep}^{c_2}$ are analytical continuations of each other on $\CC P^1\times\DD_r,$ yielding a well defined and holomorphic global change of coordinates $\widehat{\Psi}_{\ep}$ over the divisor which is, by construction, a local equivalence between $\Ff_{\omega_{\ep}}$ and $\Ff_{\widehat{\omega}_{\ep}}$  around $\Sb^1\times\CC.$    It depends holomorphically on $\ep\in V$ by the analytic dependence on initial conditions of a differential equation.  Let $\Psi_{\ep}$ stand for this diffeomorphism in $(x,y)$ variables.    Since the Riemann sphere $\CC P^1$ retracts to the origin, the equivalence $\Psi_{\ep}$ is defined on $(\DD_{r}\times\DD_{r})\backslash \{(0,0)\}$ and is analytic there, because the monoidal map is an isomorphism away from the exceptional divisor.    By Hartogs Theorem, $\Psi_{\ep}$ can be holomorphically extended until the origin.

Inasmuch as the equivalence $\Psi_{\ep}$ is constructed by lifting paths, and both the holonomy and the conjugacy $\Hhh_{\ep}$ are real (when $\ep\in\RR)$, the change of coordinates $\Psi_{\ep}$ is real as well.

\section{Realization of an admissible family.}\label{rrffweo1}
A first change of coordinates on the complexified family \eqref{weakfocus0}, depending analytically on small values of the parameter, allows to get rid of all cubic terms except for the resonant one (Poincar\'e normal form).    The weak focus is of order one if the real part of the coefficient of the third order resonant monomial is non null.     The sign $s=\pm 1$ of such a coefficient defines two different cases which are not equivalent by real equivalence.    In fact, $s$ is an analytic invariant of the system.   An analytic change of coordinates (cf. \cite{wacr}) brings the Poincar\'e map to the ``prepared'' form
\begin{equation}
\begin{array}{lll}
\Pp_{\ep}(x)= x + x(\ep+s x^2)(2\pi+O(\ep)+O(x)),
\end{array}\label{pforPpp}
\end{equation}
with multiplier $\exp(2i\pi)$ at the origin.

\begin{proposition}\label{formal-class-teo}
A germ of generic real analytic family of differential equations unfolding a germ of real analytic weak focus of order one, is formally orbitally equivalent to:
\begin{equation}
\begin{array}{lll}
\dot x &=& x(i+(\ep\pm u)(1 - A(\ep)u))\\
\dot y &=& y(-i+(\ep\pm u)(1 - A(\ep)u))
\end{array}\label{formclawf}
\end{equation} 
with $u=xy,$ for some family of constants $A(\ep)$ which is real on $\ep\in\RR$ and $A(0)\neq 0.$   The parameter $\ep$ of the formal normal form \eqref{formclawf} is called the ``canonical parameter''.
\end{proposition}
\begin{proof}
Consider the case $s = + 1.$    By a formal change of coordinates we bring the system to the form:
\begin{equation}
\begin{array}{lll}
\dot x &=& x(i+\ep - \sum_{j\geq 1} A_j(\ep)u^j):=P(x,y)\\
\dot y &=& y(-i+\ep - \sum_{j\geq 1}\overline{ A_j(\overline{\ep})}u^j):=Q(x,y)
\end{array}\label{ForClas}
\end{equation} 
where $Re ( A_1)\neq 0.$    In order to simplify the form, we iteratively use changes of coordinates $(x,y)=(\x(1+cU^n),\y(1+\overline cU^n))$ for $n\geq 1.$    Such a change allows to get rid of the term $ A_{n+1}U^{n+1}$ provided that $n+1>2.$     When $n=1$ it allows to get rid of $iIm(A_2U^2).$    Indeed, the constant $c$ must be chosen so as to verify $ A_1(c+\overline c)-nc(A_1+\overline{A_1})= A_{n+1},$ which is always solvable in $c$ as soon as $Re(A_1)\neq 0$ and $n>1.$    However, when $n=1$ we get $A_1(c+\overline c)-nc(A_1+\overline{ A_1})=A_1\overline c-\overline{ A_1}c=2iIm( A_1\overline c)\in i\RR.$    Hence, in that only case, the equation $ A_1(c+\overline c)-nc( A_1+\overline{ A_1})=iIm( A_{n+1})$ is solvable in $c.$
Finally, one divides \eqref{ForClas} by $\frac{yP-xQ}{2ixy}.$    This brings all the $Im(A_j)$ to $0.$    Then we repeat the procedure above with $c$ real to remove all higher terms in $u^j$ except for the term in $u^2.$   The cases $s=-1$ is analogous.
\end{proof}

It is easily seen that the multiplier at the origin of the Poincar\'e map of the field \eqref{formclawf} is equal to $\exp(2\pi\ep),$ so that the canonical parameter is also an analytic invariant of the Poincar\'e map.\\

\noindent{\emph{Admissible families of holomorphic germs.}}\\

Consider the germ of a holomorphic family $\Qq_{\ep}$ unfolding the germ of a codimension one analytic resonant diffeomorphism $\Qq$ with multiplier equal to $-1$ at the origin.    The formal normal form $\Qq_{0,\ep}$ of $\Qq_{\ep}$ is the semi-Poincar\'e map (or semi-monodromy) of the vector field \eqref{formclawf}, namely $\Qq_{0,\ep}=\Ll_{-1}\circ\tau_{\ep}^{\pi},$ where $\tau_{\ep}^{\pi}$ is the time $\pi$-map of the equation:
\begin{equation}
\begin{array}{lll}
\dot w = \frac{w(\ep\pm w^2)}{1+A(\ep)w^2}
\end{array}\label{semiFFNO}
\end{equation} 
and $\Ll_{-1}: w\mapsto -w.$

\begin{lemma}\label{fnor-Q}
Let $\Qq_{\ep}$ be a prepared family $(i.e.$ such that $\Qq_{\ep}^{\circ 2}$ has the form \eqref{pforPpp}$)$ unfolding a codimension one resonant diffeomorphism $\Qq$ with multiplier equal to $-1,$ and let $\Qq_{0,\ep}$ be its formal normal form, with same canonical parameter $\ep.$    Then, for any $N\in\NN^*$ there exists a real family of germs of diffeomorphisms $f_{\ep}$ tangent to the identity such that:
\begin{equation}
\begin{array}{lll}
\Qq_{\ep}\circ f_{\ep} - f_{\ep}\circ\Qq_{0,\ep} = O(x^{N+1}(\ep \pm x^2)^{N+1}).
\end{array}\label{fnor-QE}
\end{equation}
\end{lemma}
\begin{proof}
The proof is a slight modification of Theorem 6.2 in \cite{ccrr}, being given that the preparation of the family of diffeomorphisms is slightly different as well.
\end{proof}

Any germ of family of holomorphic diffeomorphisms $\Qq_{\ep}:(\CC,0)\to(\CC,0)$ verifying the hypotheses of Lemma \ref{fnor-Q}, is said to be ``admissible''.

\begin{theorem}\label{desusa}
Let $\Qq_{\ep}:(\CC,0)\to(\CC,0)$ be a real analytic family in the class of admissible germs of families, with coefficients $c_k(\ep)$ depending analytically on the canonical parameter $\ep,$ and such that $2c_2(\ep)^2+c_3(\ep)(1+c_1(\ep)^2)\neq 0$ for all $\ep\in V.$    Then the second iterate $\Qq_{\ep}^{\circ 2}$ is the monodromy  of an elliptic generic family \eqref{weakfocus0} of order one.
\end{theorem}

\section{Proof of Theorem \ref{desusa}.}
The proof is achieved in several steps.
\subsection{Family of abstract manifolds.}    By Lemma \ref{fnor-Q}, $\Qq_{\ep}$ decomposes as
\begin{equation}
\begin{array}{lll}
\Qq_{\ep} &=& (id+g_{\ep})\circ\Qq_{0,\ep}\\
&=& \Qq_{0,\ep}\circ(id+\widehat{g}_{\ep}),
\end{array}\label{eq.decomp}
\end{equation} 
where $g_{\ep}$ and $\widehat{g}_{\ep}:=\Qq_{0,\ep}^{\circ -1}\circ g_{\ep}\circ\Qq_{0,\ep}$ are $(N+1)$-flat in $x$ at the origin: $g_{\ep}(x),\widehat{g}_{\ep}(x)=O(x^{N+1}(\ep\pm x^2)^{N+1})),$ for a large integer $N\in\NN.$

We recall that the standard monoidal map endows the blow-up space with coordinates $(X,y)$ and $(x,Y),$ and the transition between them is noted $\varphi.$\\

Let $v_{0,\ep}^{1}$ be the formal normal form given by the pullback of \eqref{formclawf} in $(X,y)$ variables (with linear part $2iX\frac{\partial}{\partial X} + (\ep-i)y\frac{\partial}{\partial y})$ and let $\Ff_{v_{0}^{1}}$ be its foliation on the product $\CC^*\times \DD_y,$ where $\DD_y$ is the standard unit disk of the $y$ axis.     Consider the region $$\widetilde K_{1} =\Big\{\widetilde X\in Cov(\CC^*): -\pi/4<\arg(\widetilde X)<2\pi +\pi/4\Big\}$$ in the covering space $Cov(\CC^*)$ of the exceptional divisor, see Figure \ref{dom-rie}.     
\begin{figure}[!h]
\begin{center} 
{\psfrag{i}{\Huge{$\infty$}}
\psfrag{0}{\Huge{$0$}}
\psfrag{S'}{\Huge{$S$}}
\psfrag{S}{\Huge{$S'$}}
\psfrag{ec}{\Huge{$\simeq$}}
\scalebox{0.30}{\includegraphics{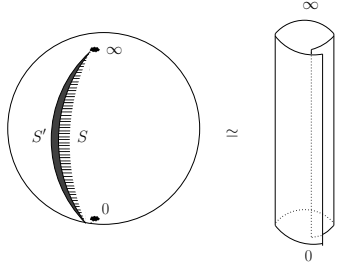}}
}
\end{center}
\caption{\label{dom-rie} \small{The domain of $\widetilde X$ in the covering space $Cov(\CC^*).$}}
\end{figure}
The pullback of $v_{0,\ep}^{1}$ by the covering map $\pi_{1}:\widetilde K_{1}\times\DD_y\to\CC^*\times\DD_y,$ defines a field $\widetilde v_{\ep}^{1}(\widetilde X,w)$ and a foliation $\widetilde{\Ff}_{v^1}$ on the product $\widetilde{M}=\widetilde{K}_{1}\times\DD_y.$   The leaves of $\widetilde{\Ff}_{v^1}$ around the \emph{flaps} 
\begin{equation*}
\begin{array}{lll}
S'_{1} &=& \{\widetilde X'\in\widetilde{K}_{c}: -\pi /4<\arg(\widetilde X')<\pi / 4\}\\
S_{1} &=& \{\widetilde X\in\widetilde{K}_{1}: 2\pi-\pi /4<\arg(\widetilde X)<2\pi + \pi /4\}
\end{array}
\end{equation*} 
are identified by means of a \emph{sealing map} $\Upsilon_{\ep}:S'_{1}\times\DD_y\to S_{1}\times\CC,$ which preserves the first coordinate and respects $\widetilde{\Ff}_{v^1}.$    It is constructed as follows.     For small values of $y,$ the holonomy map $h_{\ep,X}:\{X\}\times\DD_y\to\{1\}\times\DD_y$ along the leaves of $\Ff_{v_{0}^{1}}$ is covered by two holonomy maps, $h_{\ep,\widetilde X'}:\{\widetilde{X'}\}\times\DD_y\to \Sigma'\times\DD_y$ and $h_{\ep,\widetilde X}:\{\widetilde X\}\times\DD_y\to \Sigma\times\DD_y$ along the leaves of $\widetilde{\Ff}_{v^1}.$    The holonomies is negatively (resp. positively) oriented and noted $h_{\ep}^-$ (resp. $h_{\ep}^+),$ when $Im(X)>0$ (resp. $Im( X)<0).$    The convention:
\begin{equation}
\begin{array}{lll}
\lim_{\widetilde X\to\widetilde 1} h_{\ep,\widetilde X}^+=id
\end{array}\label{lim--zo}
\end{equation}
will be taken into account as well.   Then $\Upsilon_{\ep}(\widetilde{X}',y)=(\widetilde{X},\Delta_{\ep}(\widetilde{X}',y)),$ where
\begin{equation}
\begin{array}{lll}
\Delta_{\ep}(\widetilde{X}',y)=(h_{\ep,\widetilde X}^+)^{\circ -1}\circ(id+g_{\ep})\circ h_{\ep,\widetilde X'}^+(y),
\end{array}\label{eq.sealing}
\end{equation}
with $\pi_1(\widetilde X')=\pi_1(\widetilde X).$   The map $\Upsilon_{\ep}$ is well defined and real analytic on its image for $r>0$ small, and it depends analytically on the parameter.  Thus, it may be analytically extended to a larger domain $\{\widetilde X\in\widetilde{K}_{1} : -\pi /4<\arg(\widetilde X)<\pi\}\times\DD_y.$

Around a region of the covering of the $(x,Y)$ chart, things are naturally defined by means of the transition $\varphi.$    In particular, the family $\Dosilon_{\ep}=\varphi^*\Upsilon_{\ep}$ is a sealing map $\Dosilon_{\ep}(\widetilde Y',z)=(\widetilde Y,\nabla_{\ep}(\widetilde Y',z))$ with
\begin{equation}
\begin{array}{lll}
\nabla_{\ep}(\widetilde Y',z)=(\ell_{\ep,\widetilde Y}^+)^{\circ -1}\circ(id+\widehat{g}_{\ep})\circ\ell_{\ep,\widetilde Y'}^+(z),
\end{array}\label{Dosilon}
\end{equation}
and the transition $\varphi$ defines a field $\widetilde v_{\ep}^{2}(x,\widetilde Y),$ and foliation $\widetilde{\Ff}_{v^2}$ on the product $\widetilde{N}:=\varphi^*\widetilde{M}$ (the latter endowed with complex coordinates $(x,\widetilde Y)$ defined in the natural way).   Here, $\ell_{\ep}^{\pm}$ are the holonomies along the leaves of $\widetilde{\Ff}_{v^2}.$   The map $\Dosilon_{\ep}$ is real analytic on its image.\\

As the sealing $(\Upsilon_{\ep},\Dosilon_{\ep})$ is canonically defined on the divisor, it defines a sealing family noted $\Gamma_{\ep}:\widetilde{\Mm}\to\widetilde{\Mm},$ where $\widetilde{\Mm}$ is the pullback of $(\widetilde{M},\widetilde{N})$ by the inverse of the monoidal map.    The vector fields $(\widetilde v_{\ep}^{1},\widetilde v_{\ep}^{2})$ and foliations $(\widetilde{\Ff}_{v^1},\widetilde{\Ff}_{v^2})$ induce a vector field $\widetilde v_{\ep}$ and a foliation $\widetilde{\Ff}_{\ep}$ on $\widetilde{\Mm},$ and the coordinates on the latter are $(x,y).$    Moreover, $\Gamma_{\ep}$ is a germ of real analytic family of diffeomorphisms that preserves the transversal fibers $\Sigma_{\mu}=\{x=\mu y,\mu\in\CC^*\},$ and $(\Gamma_{\ep})_*\widetilde v_{\ep} = \widetilde v_{\ep},$ so that $\Gamma_{\ep}$ respects $\widetilde{\Ff}_{\ep}.$    Then, the quotient $$\Mm_{\ep}=\widetilde{\Mm} / \Gamma_{\ep}$$ is well defined and the vector field $\widetilde v_{\ep}$ induces a vector field $v_{\ep}$ and a foliation $\Ff_{\ep}$ on $\Mm_{\ep}.$  The leaves of this foliation project without critical points on the base $\CC^*\times\DD_y$ in the $(X,y)$ chart $(i.e$ are transversal to all lines $\{X=const.\}),$ and hence the loop generating the fundamental group of $\CC^*\times\DD_y$ defines the holonomy map of the quotient foliation $\Ff_{\ep}$ on $\Mm_{\ep}$ (for the cross section $\Sigma),$ referred to as the semi-monodromy.
\begin{proposition}\label{monod-vep}
The monodromy $\Sigma\to\Sigma$ of the field $v_{\ep}$ along the leaves of $\Ff_{\ep}$ coincides with $\Qq_{\ep}.$
\end{proposition}
\begin{proof}
The holonomy $h_{\ep,\widetilde 1}:\{\widetilde 1'\}\times\DD_y\to\{\widetilde 1\}\times\DD_y$ of $\widetilde v_{\ep}^{1}$ coincides, by construction, with the normal form $\Qq_{0,\ep}$ on $\widetilde M.$    Then, in $(x,y)$ variables, the image of the point $(x,x)\in\Sigma$ under the holonomy of $\widetilde{v}_{\ep}$ (for the section $\Sigma)$ is given by $(\Qq_{0,\ep}(x),\Qq_{0,\ep}(x))\in\Sigma.$    In addition,
\begin{equation*}
\begin{array}{lll}
\Gamma_{\ep}(\Qq_{0,\ep}(x),\Qq_{0,\ep}(x)) &=& (\Delta_{\ep}(1,\Qq_{0,\ep}(x)),\Delta_{\ep}(1,\Qq_{0,\ep}(x)))\\
&=& ((id+g_{\ep})\circ\Qq_{0,\ep}(x),(id+g_{\ep})\circ\Qq_{0,\ep}(x))\\
&=& (\Qq_{\ep}(x),\Qq_{\ep}(x))\in\Sigma,
\end{array}
\end{equation*}
where the second equality comes after \eqref{lim--zo}.
\end{proof}

\subsection{Integrability on $H_{\ep}(\Mm_{\ep})$}
In $(\widetilde X,y)$ coordinates, we introduce a smooth \emph{real} nonnegative cutoff function $\chi$ \emph{depending only on the argument of} $\widetilde X:$
\begin{equation*}
\chi(\arg\widetilde X)=\left \{ \begin{array}{lll}
1,\quad \arg\widetilde X\in(-\pi /4,\pi /4],\\
0,\quad \arg\widetilde X'\in(\pi,2\pi+\pi /4].
\end{array}
\right.
\end{equation*}
An ``identification map'' $\widetilde H_{\ep}^{1}$ is defined on $\widetilde M:$ 
\begin{equation}
\begin{array}{lll}
\widetilde H_{\ep}^{1} : (\widetilde X, y)\mapsto(\widetilde X, y + \chi(\arg\widetilde X)\{\Delta_{\ep}(\widetilde X,y)-y\}),
\end{array}\label{id-mapZ}
\end{equation}
for $c_1,c_2$ the monoidal map in charts \eqref{bdown}.    Notice that $\widetilde{H}_{\ep}^{1}|_{S'_{1}\times\DD_y} \equiv (id_X,\Delta_{\ep})$ and $\widetilde{H}_{\ep}^{1}|_{S_{1}\times\DD_y} \equiv  (id_X,id_y),$ and so this map respects the sealing $\Upsilon_{\ep}.$\\

In $(x,y)$ variables the function $\chi$ yields a real smooth map $\widehat{\chi}(x,y)=\chi(\arg(x/y))$ which depends only on the argument of the quotient $x /y,$ and the blow down of \eqref{id-mapZ} in $(x,y)$ coordinates equips the target space with coordinates $(z,w):$
\begin{equation}
\begin{array}{lll}
(z,w) = \widetilde{H}_{\ep}(x,y)\\
= (x+\widehat{\chi}(x,y)\{\nabla_{\ep}\circ c_2^{-1}(x,y)-x\},w+\widehat{\chi}(x,y)\{\Delta_{\ep}\circ c_1^{-1}(x,y)-y\}).
\end{array}\label{H_0--}
\end{equation}
By definition, $\widetilde{H}_{\ep}$ induces an ``identification family'' in the quotient: $$H_{\ep}:\Mm_{\ep}\to\CC^2.$$    For every fixed $\ep,$ the latter is a real analytic diffeomorphism which endows the target space with an almost complex structure induced from the standard complex structure on $\Mm_{\ep},$ as shown later.    In addition, it depends analytically on the parameter.  If the function $g$ in \eqref{eq.decomp} is $(N+1)$-flat at $x = y = 0,$ then \eqref{H_0--} is infinitely tangent to the origin:
\begin{proposition}\label{funcgH_0-ac}
The maps \eqref{eq.sealing} and \eqref{Dosilon} admit the asymptotic estimates
\begin{equation}
\begin{array}{lll}
|\Delta_{\ep}\circ c_1^{-1}(x,y)-y| &=& O(|x|^{\frac{N}{2}(1-\ep_2)}|y|^{\frac{N}{2}(1+\ep_2)+1})\\
|\nabla_{\ep}\circ c_2^{-1}(x,y)-x| &=& O(|x|^{\frac{N}{2}(1+\ep_2)+1}|y|^{\frac{N}{2}(1-\ep_2)})
\end{array}\label{a-est--}
\end{equation} 
in the bidisk $\DD_{x}\times\DD_{y},$ where $\ep=\ep_1+i\ep_2.$
\end{proposition}
\begin{proof}
In $(\widetilde X,y)$ variables, the following estimate for the holonomy map $h_{\ep,\widetilde X}:\{\widetilde X\}\times\DD_{w}\to\{\widetilde 1\}\times\CC$ is well known:
\begin{equation*}
\begin{array}{lll}
e^{-M|\lambda(\ep)(\widetilde X-1)|-\frac{\ep_1\arg\widetilde X}{2}}|\widetilde X|^{\frac{1-\ep_2}{2}}|y|\leq |h_{\ep,\widetilde X}(y)|\leq e^{M|\lambda(\ep)(\widetilde X-1)|-\frac{\ep_1\arg\widetilde X}{2}}|\widetilde X|^{\frac{1-\ep_2}{2}}|y|,
\end{array}
\end{equation*} 
where $M=M(\widetilde X,y)<\infty$ is a positive constant depending on a bound for the nonlinear part of the foliation along the segment with endpoints $\widetilde X,1,$ and $\lambda(\ep)=(\ep-i) / 2i$ is the ratio of eigenvalues in $(X,y)$ chart.  By \eqref{eq.decomp}, $$h_{\ep,\widetilde X}^{-1}\circ (id+g)\circ h_{\ep,\widetilde X'}=h_{\ep,\widetilde X}^{-1}\circ(h_{\ep,\widetilde X'}+g\circ h_{\ep,\widetilde X'})=id+O(|\widetilde X|^{\frac{N}{2}(1-\ep_2)}|y|^{N+1}).$$   In $(x,Y)$ coordinates, the estimate is obtained by symmetry.  Since $x=\widetilde Xy$ and $y=\widetilde Yx,$ the conclusion follows.
\end{proof}
\begin{corollary}\label{gtangid}
The family $H_{\ep}$ is tangent to the identity.
\end{corollary}

The pullback of the complex structure on $\Mm_{\ep}$ by the map $H_{\ep}^{-1}$ is an almost complex structure defined by the pullback of the $(1,0)$-subbundle on $\Mm_{\ep},$ which is spanned by 
\begin{equation}
\begin{array}{lll}
\widetilde{\zeta}_{1,\ep} = dz = d(x+\widehat{\chi}\cdot\{\nabla_{\ep}\circ c_2^{-1}-x\}),\\
\widetilde{\zeta}_{2,\ep} = dw = d(y+\widehat{\chi}\cdot\{\Delta_{\ep}\circ c_1^{-1}-y\}),
\end{array}\label{diforms}
\end{equation} 
on $\widetilde H_{\ep}(\widetilde{\Mm}).$   The forms $d(\nabla_{\ep}\circ c_2^{-1})$ and $d(\Delta_{\ep}\circ c_1^{-1})$ are holomorphic on their domains and $\widetilde{\zeta}_{1,\ep}$ and $\widetilde{\zeta}_{2,\ep}$ have two different sectorial representatives: 
\begin{equation}
\left.
\begin{array}{lll}
\widetilde{\zeta}_{1,\ep}=\left \{ \begin{array}{lll}
\zeta_{1,\ep}^0=dx, & |\arg x-\arg y-13\pi / 8|< 5 \pi/8, \\
\zeta_{1,\ep}^1=d(\nabla_{\ep}\circ c_2^{-1}), & |\arg x-\arg y|< \pi /4,
\end{array}
\right.\\
&\\
\widetilde{\zeta}_{2,\ep}=\left \{ \begin{array}{lll}
\zeta_{2,\ep}^1=d(\Delta_{\ep}\circ c_2^{-1}), & |\arg x-\arg y|< \pi /4\\
\zeta_{2,\ep}^0=dy, & |\arg x-\arg y-13\pi / 8|< 5 \pi/8,
\end{array}
\right.
\end{array}
\right.
\end{equation}
so that $\zeta_{1,\ep}^1=\Gamma_{\ep}^*\zeta_{1,\ep}^0$ and $\zeta_{2,\ep}^1=\Gamma_{\ep}^*\zeta_{2,\ep}^0.$    Thus they yield forms $\zeta_{1,\ep}$ and $\zeta_{2,\ep}$ on $\Mm_{\ep}.$    The almost complex structure induced on $H_{\ep}(\Mm_{\ep})\subset\CC^2$ by the complex structure on $\Mm_{\ep}$ is defined by the two forms 
\begin{equation}
\begin{array}{lll}
\omega_{1,\ep}=(H_{\ep}^{-1})^*\zeta_{1,\ep},\quad\quad\omega_{2,\ep}=(H_{\ep}^{-1})^*\zeta_{2,\ep}.
\end{array}\label{alcomplexstructure}
\end{equation}

\begin{lemma}\label{estimateforms}
Let $\delta$ be a small positive number with $|\ep|<\delta.$    If $\alpha$ and $\beta$ are the orders of flatness in $x$ and $y$ (resp. $y$ and $x)$ of the difference $\omega_{1,\ep}-dx$ (resp. $\omega_{2,\ep}-dy),$ then the form $\omega_{1,\ep}$ (resp. $\omega_{2,\ep})$ can be extended as $dx$ (resp. $dy)$ along the $x$-axis (resp. $y$-axis) until the order $\alpha$ if the number $N$ in \eqref{eq.decomp} is sufficiently large so as to ensure
\begin{equation}
\begin{array}{lll}
N >\max\left\{\frac{2(\alpha-1)}{1-\delta},\frac{2\beta}{1-\delta}\right\}.
\end{array}\label{N-beta}
\end{equation}
\end{lemma}
\begin{proof}
By \eqref{alcomplexstructure}, it suffices to study the difference $$\widetilde{H}_{\ep}(x,y)-(x,y)=(\widehat{\chi}(x,y)\{\nabla_{\ep}\circ c_2^{-1}(x,y)-x\},\widehat{\chi}(x,y)\{\Delta_{\ep}\circ c_1^{-1}(x,y)-y\}).$$   The definition of $\xi$ yields
\begin{equation}
\begin{array}{lll}
\left|\frac{\partial^{i+j}\widehat{\chi}}{\partial x^p\partial\overline x^q\partial y^r\partial\overline y^s}\right|<C^{\underline{st}}\cdot\frac{\MB_{i+j}}{|x|^i|y|^j}
\end{array}\label{est-X}
\end{equation} 
for all $i=p+q \in\NN,$ $j=r+s \in\NN$ and $\MB_{i+j}:= 
\max_{\substack{
		0\leq k\leq i+j\\
		\theta\in I}} 
|\chi^{(k)}(\theta)|$ with $I=[-\pi /4,2\pi+\pi /4].$    To lighten the notation, put $f(x,y)=\nabla_{\ep}\circ c_2^{-1}(x,y)-x.$  Proposition \ref{funcgH_0-ac} implies that for all $k,l\in\NN,$ there exists a real constant $L=L(N,\alpha,\beta)>0$ such that
\begin{equation}
\begin{array}{lll}
\left|\frac{\partial^{\alpha+\beta}(\widehat{\chi}\cdot f)}{\partial x^p\partial\overline{x}^q\partial y^r\partial\overline{y}^s}\right| \leq L\cdot |x|^{\frac{N}{2}(1+\ep_2)+1-\alpha} \cdot |y|^{\frac{N}{2}(1-\ep_2)-\beta},
\end{array}\label{est-prod}
\end{equation}
for $\alpha=p+q$ and $\beta=r+s.$   Hence, if $|\ep|<\delta<<1$ and the order $N$ of  $g_{\ep}$ satisfies \eqref{N-beta} then the left hand side of \eqref{est-prod} tends to zero uniformly in $|x|<1,$ and thus $\omega_{1,\ep}$ and $dx$ coincide until the order $\alpha$ along the $x$-axis.    The assertion for the difference $\omega_{2,\ep}-dy$ follows by duality.
\end{proof}
The set $H_{\ep}(\Mm_{\ep})\subset\CC^2$ does not contain the axes of coordinates: its closure is $C^{\infty}$-diffeomorphic to a closed neighborhood of the origin of $\CC^2.$   Lemma \ref{estimateforms} shows that the almost complex structure generated by \eqref{alcomplexstructure} on $H_{\ep}(\Mm_{\ep})$ can be extended as $\omega_{1,\ep}=dx$ along the $x$-axis, and as $\omega_{2,\ep}=dy$ along the $y$-axis, until a well-defined order.    This almost complex structure is integrable.  Indeed, $\omega_{1,\ep}$ is obtained from the pullback of $\zeta_{1,\ep}$ and since the forms $d(\nabla_{\ep}\circ c_2^{-1})$ and $d(\Delta_{\ep}\circ c_1^{-1})$ are holomorphic on their domains and $\widehat{\chi}$ is of class $C^{\infty},$ $d\widetilde{\zeta}_{1,\ep}$ contains no forms of type $(0,2).$   By symmetry, the same holds for $d\widetilde{\zeta}_{2,\ep}.$  If $L^{1,0}$ is the span of the forms $\omega_{1,\ep},\omega_{2,\ep},$ then this integrability condition holds for $L^{1,0}$ on the surface $H_{\ep}(\Mm_{\ep}),$ and by continuity it remains valid after extension until the axes.   Hence, for each $\ep\in V$ the Newlander-Nirenberg Theorem ensures the existence of a smooth chart $\widetilde{\Lambda}_{\ep}=\widetilde{\Lambda}_{\ep}(z,w),$
\begin{equation}
\begin{array}{lll}
\widetilde{\Lambda}_{\ep}=(\widetilde{\xi}_{\ep}^1,\widetilde{\xi}_{\ep}^2):\widetilde H_{\ep}(\widetilde{\Mm})\to\CC^2,
\end{array}\label{csmocht}
\end{equation} 
which is holomorphic in the sense of the almost complex structure \eqref{diforms}.   It induces, in turn, the germ of a family of smooth charts 
\begin{equation}
\begin{array}{lll}
\Lambda_{\ep}=(\xi_{\ep}^1,\xi_{\ep}^2):\BB(r)\subset H_{\ep}(\Mm_{\ep})\to\CC^2
\end{array}\label{jodfellout}
\end{equation} 
in the quotient, where $\BB(r)$ is a small ball around the origin.    This chart is, by definition, holomorphic in the sense of the extended almost complex structure \eqref{alcomplexstructure}.

\begin{theorem}\label{NNext}
The germ of smooth charts $\Lambda_{\ep}$ respects the real foliation, is tangent to the identity at the origin, and depends analytically on the parameter.
\end{theorem}
\begin{proof}
In order to show that the chart respects the real foliation, it suffices to prove that $\widetilde{\Lambda}_{\ep}$ is real, namely, it sends $\{z=\overline w\}\simeq\RR^2$ into $\RR^2\subset\CC^2$ when $\ep\in\RR.$

The family of diffeomorphisms \eqref{H_0--} is analytic with respect to the structure \eqref{diforms}.    It follows that, modulo a linear combination,  
\begin{equation*}
\begin{array}{lll}
dz &=& dx + e_{1,\ep}^1d\overline{x} + e_{2,\ep}^1d\overline{y}\\
dw &=& dy + e_{1,\ep}^2d\overline{x} + e_{2,\ep}^2d\overline{y}, 
\end{array}
\end{equation*} 
where the coefficients are computed in terms of $\chi,\Delta,\nabla$ and its derivatives, and satisfy:
\begin{equation}
\begin{array}{lll}
\overline{e_{j,\overline{\ep}}^k(\overline y,\overline x)} &=& e_{k,\ep}^j(x,y),\quad j,k\in\{1,2\}.
\end{array}\label{funcSr-}
\end{equation}
This is because $\RR^2$ is itself invariant under \eqref{H_0--}: $H_{\ep}(\{x=\overline y\})\subset\RR^2$ when the parameter is real.   By Proposition \ref{funcgH_0-ac}, $e_{j,\ep}^k=\frac{o(1)}{1+o(1)},$ yielding:
\begin{equation}
\begin{array}{lll}
e_{j,\ep}^k(0,0)=0,\quad j,k\in\{1,2\}.
\end{array}\label{eij=0}
\end{equation}
Suppose that the image $\widetilde H_{\ep}(\Mm)$ contains a small bidisk $\DD_s\times\DD_s,$ and write $G_{\ep}:=\widetilde H_{\ep}^{-1}.$    Consider the pullback $\aG_{j,\ep}^k=G_{\ep}^*(e_{j,\ep}^k):\DD_s\times\DD_s\to\CC^2$ given by
\begin{equation*}
\begin{array}{lll}
\aG_{j,\ep}^k(z,w)=G_{\ep}^*(e_{j,\ep}^k)(z,w)\equiv e_{j,\ep}^k(G_{\ep}(z,w)),\quad j,k=1,2, \ \ \ep\in V,
\end{array}
\end{equation*} 
for $(z,w)\in\DD_s\times\DD_s.$    By \eqref{funcSr-} the collection $\aG_{j,\ep}^k$ satisfies again:
\begin{equation}
\begin{array}{lll}
\overline{\aG_{j,\overline{\ep}}^k(\overline{w},\overline{z})} &=& \aG_{k,\ep}^j(z,w),
\end{array}\label{funcSr}
\end{equation}
and by \eqref{eij=0}, $\aG_{j,\ep}^k(0,0)=0.$\\

\noindent{\emph{Notation.}}   We will write $z^1=z, \ \ z^2=w,$ and:
\begin{equation*}
\begin{array}{lll}
\partial_j=\frac{\partial}{\partial z^j},\quad \overline{\partial}_j=\frac{\partial}{\partial\overline{z^j}},\quad j=1,2.
\end{array}
\end{equation*}
A complex valued function $\xi$ such that:
\begin{equation}
\begin{array}{lll}
\overline{\partial}_j\xi-(\aG_j^1\partial_1\xi+\aG_j^2\partial_2\xi)=0,\quad j=1,2
\end{array}\label{compOpee}
\end{equation} 
is called (cf. \cite{NNir}) holomorphic with respect to the given almost complex structure.    Instead of considering the new coordinates \eqref{csmocht} as solutions to \eqref{compOpee} and functions of $(z,w)$ and their complex conjugates, the coordinates $(z,w)$ are supposed to be functions of \eqref{csmocht} and their complex conjugates.    Inasmuch as it suffices to study only the real character of the chart $\widetilde{\Lambda}_{\ep},$ the tildes on the chart $(\widetilde{\xi}_{\ep}^1,\widetilde{\xi}_{\ep}^2)$ are dropped from now on.\\

\noindent{\emph{Notation.}}   The holomorphic and antiholomorphic dual differentials are:
\begin{equation}
\begin{array}{lll}
d_{j,\ep}=\frac{\partial}{\partial\xi_{\ep}^j},\quad\overline{d}_{j,\ep}=\frac{\partial}{\partial\overline{\xi_{\ep}^j}},\quad j=1,2.
\end{array}\label{boundOpe}
\end{equation}
\noindent It is known (cf. \cite{NiWoo}, pp. 445) that for every $\ep\in V,$ the map $G_{\ep}$ from $\DD_s\times\DD_s\subset\CC^2$ to the almost complex manifold $\widetilde{\Mm}$ is holomorphic if and only if its coordinates $(z,w)=G_{\ep}^*(x,y)$ satisfy the differential equations
\begin{equation}
\begin{array}{lll}
\overline{d}_{j,\ep} z^k+\aG_{m,\ep}^k\overline{d}_{j,\ep}\overline{z}^m=0,\quad j,k=1,2.
\end{array}\label{non-linear}
\end{equation}
In such a case, \eqref{compOpee} yields:
\begin{equation}
\begin{array}{lll}
\overline{d}_{j,\ep}\xi_{\ep}^{p} &=& \partial_k\xi_{\ep}^{p} \overline d_{j,\ep}z^k+\overline{\partial}_k\xi_{\ep}^{p}\overline{d}_{j,\ep}\overline{z}^k\\
&=& \partial_k\xi_{\ep}^{p}\{\overline{d}_{j,\ep} z^k+\aG_{m,\ep}^k\overline{d}_{j,\ep}\overline{z}^m\}\\
&=& 0
\end{array}\label{eqtodvar}
\end{equation}
for $j=1,2.$    Notice that the replacement of \eqref{non-linear} in the term after the first equality of \eqref{eqtodvar}, yields: 
\begin{equation*}
\begin{array}{lll}
\overline{d}_{j,\ep}\xi_{\ep}^{p}=\overline{d}_{j,\ep}\overline{z}^k\{\overline{\partial}_k\xi_{\ep}^{p}-\aG_{k,\ep}^i\partial_i\xi_{\ep}^{p}\},\quad p=1,2.
\end{array}
\end{equation*} 
Thus the parametric Cauchy-Riemann equations $\overline{d}_{j,\ep}\xi_{\ep}^{p}=0$ are equivalent to the system \eqref{compOpee} if $z,w$ satisfy \eqref{non-linear} with the matrix $[\overline{d}_{j,\ep}\overline{z}^k]$ non-singular for all $\ep$ in the symmetric neighborhood $V.$    We find real solutions to \eqref{non-linear}.\\

Denote by $T^1,T^2$ the integral operators
\begin{equation}
\begin{array}{lll}
T^1f(z,w)=\frac{1}{2i\pi}\iint_{|\tau|<\rho} \frac{f(\tau,w)}{z-\tau}d\overline{\tau}d\tau,\\
T^2f(z,w)=\frac{1}{2i\pi}\iint_{|\tau|<\rho} \frac{f(z,\tau)}{w-\tau}d\overline{\tau}d\tau,\\
\end{array}\label{intOperators}
\end{equation} 
with $\rho>0$ fixed and $f=f(z,w)$ has suitable differentiability properties and, eventually, depends on additional complex coordinates.    A short calculation shows that if $f_1,f_2$ are as above and $f_1(z,w)=\overline{f_2(\overline{w},\overline{z})},$ then 
\begin{equation}
\begin{array}{lll}
T^1f_1(z,w)=\overline{T^2f_2(\overline{w},\overline{z})}.
\end{array}\label{navsupp}
\end{equation}
The non-linear differential system corresponding to \eqref{non-linear} is given by the integral equation (cf. \cite{NiWoo}):
\begin{equation}
\begin{array}{lll}
z^k(\xi_{\ep}^1,\xi_{\ep}^2)&=&\xi_{\ep}^k+\mathbf{TF}^k[z,w](\xi_{\ep}^1,\xi_{\ep}^2)-\mathbf{TF}^k[z,w](0,0),\ \ k=1,2
\end{array}\label{eqsnon-linear}
\end{equation} 
where
\begin{equation*}
\begin{array}{lll}
\mathbf{TF}^k:=T^1f_{1k}+T^2f_{2k}-\frac{1}{2}\left\{T^1\overline{d}_{1,\ep} T^2 f_{2k} + T^2\overline{d}_{2,\ep} T^1 f_{1k}\right\},\quad k=1,2\\
\end{array}
\end{equation*}
are the Nijenhuis-Woolf operators, and
\begin{equation*}
\begin{array}{lll}
f_{jk}(z,w)(\xi^1,\xi^2)(\ep)= -(\aG_{1,\ep}^k(z,w)\overline{d}_{j,\ep}\overline{z}+\aG_{2,\ep}^k(z,w)\overline{d}_{j,\ep}\overline{w}),\ \ i,j\in\{1,2\}.
\end{array}
\end{equation*}
\begin{corollary}\label{FTsymmetry}
For every $(z,w)$ in a neighborhood of the origin and for every initial value $(\xi_{\ep}^1,\xi_{\ep}^2),$ the Nijenhuis-Woolf operators are related through: 
\begin{equation}
\begin{array}{lll}
\mathbf{TF}^1[z,w](\xi^1,\xi^2)(\ep)=\overline{\mathbf{TF}^2[\overline{w},\overline{z}](\overline{\xi^2},\overline{\xi^1})(\ep)}
\end{array}\label{tehaspint}
\end{equation}
when the parameter is real.
\end{corollary}
\begin{proof}
This is plain consequence of \eqref{funcSr}, the definition of the dual differentials \eqref{boundOpe} and property \eqref{navsupp} on the maps $f_{jk}.$
\end{proof}
The pair of coordinates $(\xi_{\ep}^1,\xi_{\ep}^2)$ is referred to as the initial value of \eqref{eqsnon-linear}.   For $\ep=0,$ the system \eqref{eqsnon-linear} is solved by means of a Picard iteration process (fixed point Theorem) which converges in a small ball $\BB(r_0)$ of radius $r_0>0$ around the origin of $(z,w)$ coordinates (cf. \cite{NNir},\cite{NiWoo}).    It turns out that for all $|\ep|$ small and fixed, any solution to \eqref{eqsnon-linear} is well defined on $\BB(r),$ with $r=r_0/2.$    Moreover, if $r$ is small enough, then the solution $(z,w)$ is unique:
\begin{lemma}\cite{NNir}\label{2nuni}	
For $r$ sufficiently small, and $\ep\in V$ fixed, the integral system \eqref{eqsnon-linear} admits a unique solution $(z,w)$ satisfying also \eqref{non-linear} and such that the parametric transformation $\widetilde{\Lambda}_{\ep}^{\circ -1}$ from the $(\xi_{\ep}^1,\xi_{\ep}^2)$ coordinates to $(z,w)$ coordinates has non-vanishing Jacobian.
\end{lemma}

\begin{proposition}\label{lambrecha}
The chart \eqref{csmocht} respects the real foliation and is tangent to the identity. 
\end{proposition}
\begin{proof}
Let $(\xi_{\ep}^1,\xi_{\ep}^2)$ be the initial value and $(z,w)$ be the solution to \eqref{eqsnon-linear}.    If the initial condition satisfies $\xi_{\ep}^1=\overline{\xi_{\ep}^2},$ then Corollary \ref{FTsymmetry} leads to
\begin{equation*}
\begin{array}{lll}
\overline{z^k} &=& \xi_{\ep}^k+\mathbf{TF}^k[\overline{w},\overline{z}](\xi_{\ep}^1,\xi_{\ep}^2)-\mathbf{TF}^k[\overline{w},\overline{z}](0,0),\ \ k=1,2
\end{array}\label{eqsnon-linear11}
\end{equation*} 
and the unicity of the solution carries $z=\overline{w}.$    Since $\widetilde{\Lambda}_{\ep}$ has non-vanishing Jacobian, it is a local isomorphism if $r>0$ is small.    
By \eqref{eqsnon-linear}, the chart $\widetilde{\Lambda}_{\ep}$ is tangent to the identity at the origin.
\end{proof}
Inasmuch as $\Gamma_{\ep},H_{\ep}$ and \eqref{csmocht} respect the real foliation, the chart $\Lambda_{\ep}$ is real when $\ep$ is real, and is clearly tangent to the identity at the origin.    This concludes the proof of Theorem \ref{NNext}.
\end{proof}

\noindent{\emph{End of the proof of Theorem \ref{desusa}.}}    The composition $\vartheta_{\ep}=\Lambda_{\ep}\circ H_{\ep}:H_{\ep}^{-1}(\BB(r))\to\CC^2$ between complex analytic manifolds is honestly biholomorphic.    The closure $\Ww:=\overline{\vartheta_{\ep}(H_{\ep}^{-1}(\BB(r)))}$ contains the origin in its interior.    It remains to check that the family of vector fields defined on $\Ww$ by the pushforward $$\vv_{\ep}=(\vartheta_{\ep})_* v_{\ep}$$ is orbitally equivalent to a generic family unfolding a weak focus with formal normal form \eqref{formclawf}.    

\begin{proposition}
The quotient of the eigenvalues of $\vv_{\ep}$ is equal to $\frac{\ep+ i}{\ep-i}.$
\end{proposition}
\begin{proof}
By Theorem \ref{NNext}, the components $\PP_{\ep},\QQ_{\ep}$ of $$\vv_{\ep}(\x,\y)=\PP_{\ep}(\x,\y)\frac{\partial}{\partial \x} + \QQ_{\ep}(\x,\y)\frac{\partial}{\partial \y}$$ are related through \eqref{realchar} as well, and then the eigenvalues of the vector field $\vv_{\ep}$ are complex conjugate.    We call them $\tau(\ep),\overline{\tau(\ep)},$ with $\tau(\ep) = a(\ep) +ib(\ep)$ and $a(\ep),b(\ep)$ depend analytically on $\ep$ small and are real on $\ep\in\RR.$    In the $(\X,\y)$ chart of the blow-up, $\vv_{\ep}$ gives rise to a family of equations of the form:
\begin{equation*}
\begin{array}{lll}
\dot \X &=& (\tau-\overline{\tau})\X+...\\
\dot \y &=& \overline{\tau}\y+...
\end{array}
\end{equation*} 
with Poincar\'e map (cf. \cite{MM}) $\Pp_{\ep}(\y)=\exp\left(2i\pi\left(\frac{2\overline{\tau}}{\tau-\overline{\tau}}\right)\right)\y+...,$ while in $(\x,\Y)$ coordinates, $\vv_{\ep}$ gives rise to the system:
\begin{equation*}
\begin{array}{lll}
\dot \x &=& \tau \x+...\\
\dot \Y &=& (\overline{\tau}-\tau)\Y+...
\end{array}
\end{equation*} 
with Poincar\'e map $\Pp_{\ep}(\x)=\exp\left(-2i\pi\left(\frac{2\tau}{\overline{\tau}-\tau}\right)\right)\x+...$ (computed on the cross section $\x=\y).$   It is easily seen that: $$\mu(\ep)=\exp\left(2i\pi\left(\frac{2\overline{\tau}}{\tau-\overline{\tau}}\right)\right)=\exp\left(-2i\pi\left(\frac{2\tau}{\overline{\tau}-\tau}\right)\right)=\exp\left(2i\pi\left(\frac{\tau+\overline{\tau}}{\tau-\overline{\tau}}\right)\right),$$ where $\Pp_{\ep}'(0)=\mu(\ep).$    On the other hand, $\mu(\ep)=\exp(2\pi\ep)$ by the preparation \eqref{pforPpp}.    Thus, $$2\pi\ep=2\pi\frac{2a(\ep)}{2b(\ep)}+2i\pi m,$$ for some $m\in\NN.$    This means that 
\begin{equation}
\begin{array}{lll}
\frac{a(\ep)}{b(\ep)}=\ep-im.
\end{array}\label{m==0}
\end{equation} 
Inasmuch as $a(\ep),b(\ep)$ are real on $\ep\in\RR,$ the equation \eqref{m==0} implies that $m=0,$ and the conclusion follows.
\end{proof}

\section{Acknowledgements.}
I am grateful to Christiane Rousseau for suggesting the problem and supervising this work, and to Colin Christopher and Sergei Yakovenko for helpful discussions.   I am indebted to an unknown referee as well, for suggesting many important comments concerning the previous version of the paper.
\bibliography{mybib}
\end{document}